\documentclass[12pt,twoside]{amsart}
\usepackage{amsmath}
\usepackage{amsthm}
\usepackage{amsfonts}
\usepackage{amssymb}
\usepackage{latexsym}
\usepackage{mathrsfs}
\usepackage{amsmath}
\usepackage{amsthm}
\usepackage{amsfonts}
\usepackage{amssymb}
\usepackage{latexsym}
\usepackage{geometry}
\usepackage{dsfont}
\usepackage[dvips]{graphicx}
\usepackage{color}
\usepackage[all]{xy}

\usepackage{amsthm,amsmath,amssymb}\usepackage{mathrsfs}

\date{}
\allowdisplaybreaks[4] \footskip=15pt
\renewcommand{\uppercasenonmath}[1]{}

\usepackage{graphicx,amssymb}
\usepackage[all]{xy}
\usepackage{amsmath}

\allowdisplaybreaks
\usepackage{amsthm}
\usepackage{color}

\theoremstyle{plain}
\newtheorem{theorem}{Theorem}[section]
\newtheorem{proposition}[theorem]{Proposition}
\newtheorem{lemma}[theorem]{Lemma}
\newtheorem{corollary}[theorem]{Corollary}
\theoremstyle{definition}
\newtheorem{example}[theorem]{Example}
\newtheorem{definition}[theorem]{Definition}

\theoremstyle{definition}

\theoremstyle{remark}
\newtheorem{remark}[theorem]{Remark}

\newcommand{\pf}{\noindent\begin {proof}}
\newcommand{\epf}{\end{proof}}

\newcommand{\Ext}{\mbox{\rm Ext}}
\newcommand{\Hom}{\mbox{\rm Hom}}

\newcommand{\Prufer}{Pr\"{u}fer}

\newcommand{\C}{\mathcal{C}}

\newcommand{\FPR}{\mathcal{FPR}}

\newcommand{\E}{\mathcal{E}}

\newcommand{\Proj}{\mathcal{P}}

\def\fd{{\rm fd}}

\def\Hom{{\rm Hom}}
\def\Ext{{\rm Ext}}

\def\Rad{{\rm Rad}}

\def\K{{\rm K}}

\def\Nil{{\rm Nil}}

\def\E{{\rm E}}
\def\H{{\rm H}}
\def\C{{\rm \v{C}}}

\def\Max{{\rm Max}}
\def\DW{{\rm DW}}

\def\gld{{\rm gld}}
\def\fPD{{\rm fPD}}
\def\fpD{{\rm fpD}}
\def\FPD{{\rm FPD}}

\def\Spec{{\rm Spec}}

\def\pd{{\rm pd}}

\def\m{{\frak m}}
\def\M{{\frak M}}
\def\grade{{\rm grade}}
\def\depth{{\rm dp}}

\def\Nil{{\rm Nil}}
\def\p{{\frak p}}
\def\q{{\frak q}}

\begin{document}
\begin{center}
{\large  \bf  The small finitistic dimensions of commutative rings, II}

\vspace{0.5cm}   Xiaolei Zhang$^{a}$\\

{\footnotesize $a$.\ School of Mathematics and Statistics, Shandong University of Technology, Zibo 255000, China\\}
\end{center}

\bigskip
\centerline { \bf  Abstract}   The small finitistic dimension $\fPD(R)$ of a ring $R$ is defined  to be the supremum of projective dimensions of $R$-modules with finite projective resolutions.  In this paper,  we investigate  the small finitistic dimensions of four types of ring constructions:  polynomial rings, formal power series rings, trivial extensions and amalgamations. Besides, we show the small finitistic dimensions of a ring is less than or equal to its Krull dimension. We also give a total ring of quotients with infinite small finitistic dimension.
\bigskip
\leftskip10truemm \rightskip10truemm \noindent
\vbox to 0.3cm{}

\leftskip0truemm \rightskip0truemm
\bigskip
\section{introduction}
Throughout this paper, $R$ is a commutative ring with identity. Let $R$ be a ring. Denote by $\dim(R)$ the Krull dimension of $R$, $\Max(R)$ the maximal prime spectrum of $R$, $\Spec(R)$ the  prime spectrum of $R$, and $\Nil(R)$ the nil-radical of $R$.
Let $M$ be an $R$-module, we use $\pd_RM$ (resp., $\fd_RM$) to denote the projective (resp., flat) dimension of $M$ over $R$. Write $\gld(R)$ (resp., $w.\gld(R)$) for the global dimension (resp., weak global dimenison) of $R$. 

Since many classical rings, such as non-regular Noetherian local ring,   have infinite global dimensions or weak  global dimensions, Bass \cite{B60} introduced two finitistic dimensions of a ring $R$. The little (resp., big) finitistic dimension of $R$, denoted by $\fpD(R)$ (resp., $\FPD(R)$), is defined to be the supremum of the projective dimensions of  all finitely generated (resp., all) $R$-modules $M$ with finite projective dimensions.
In case $R$ is a local Noetherian ring, Auslander and Buchsbaum \cite{AB58} showed that the small finitistic dimension $\fPD(R)$ of $R$  coincides with the depth of $R$. However, there are little progress on the non-Noetherian rings since the syzygies of finitely generated modules are  not finitely generated over non-Noetherian rings in general.

To amend this gap, Glaz \cite{G89} revised the notion of little finitistic dimension of a ring $R$: $\fPD(R)$, which is called the small finitistic dimension of $R$ by Wang et al. \cite{fkxs20}, is the supremum of projective dimensions of $R$-modules with finite projective resolution (see Section 3 for more details). The studies of  small finitistic dimensions were motivated by two conjectures proposed by Glaz et al. \cite{CFFG14} who asked that is the small finitistic dimension of a \Prufer\ ring (resp., total ring of quotients) at most $1$ (resp., $0$)?   In 2020, Wang et al. \cite{wzcc20,fkxs20} characterized rings $R$ with $\fPD(R)=0$, and then gave an example of total ring of quotients with small finitistic dimensions larger than $1$ giving a negative answer to Glaz's questions. Recently, The author of this paper and Wang \cite{zw23} characterized  small finitistic dimensions in terms of  finitely generated semi-regular ideals, tilting modules, cotilting modules of cofinite type and  vaguely associated  prime ideals. Furthermore, they gave   examples of total rings of quotients $R$ with $\fPD(R)=n$ for each $n\in \mathbb{N}$.

The motivation of this paper is to give some formulas of the small finitistic dimensions of classical ring constructions. In fact, we show that the  small finitistic dimensions of polynomial rings or formal power series rings are equal to these of original rings plus $1$ under some assumptions (see Theorem \ref{poly}, Theorem \ref{scr} and Theorem \ref{pow}). We also obtain accurate formulas of finitistic dimensions of trivial extensions and amalgmations (see Theorem \ref{trivext} and Theorem \ref{amg}). Besides, we show that the small finitistic dimensions of a ring is less than or equal to its Krull dimension (see Theorem \ref{dim}). We also give a total ring of quotients with infinite small finitistic dimension (see Example \ref{glazinf}).

\section{some types of grades}
From the proof the characterizations of small finitistic dimensions given in \cite[Theorem 3.1]{zw23}, we find that it has a closely connection with the notions of  Koszul cohomology, \v{C}ech cohomology, local cohomology and their induced grades. We give a brief review on these notions in this section.

Let $\textbf{x}=x_1,\dots,x_n$ be a finite sequence of elements in ring $R$. Let $\alpha=(i_1,\dots,i_p), 1\leq i_1<\cdots<i_p\leq n$ be an ascending sequence of integers with $0\leq p\leq n$. $K_p(\textbf{x})$ is defined to be a free $R$-module on a basis $e_{\alpha}=e_{i_1}\wedge \dots\wedge e_{i_p}$. Define an $R$-homomorphism $$d_p:K_p(\textbf{x})\rightarrow K_{p-1}(\textbf{x}):\ \ d_p(e_{\alpha})=\sum\limits_{j=1}^p(-1)^{j+1}x_{i_j}e_{i_1}\wedge \dots\wedge \widehat{e_{i_j}}\wedge \dots\wedge e_{i_p},$$
where\ $\widehat{\ }$\ means deleting the item. It is easy to verify that $d_{p-1}\circ d_p=0$. So there is a finite complex, which is called Koszul complex: $$K_{\bullet}(\textbf{x}):\ \ 0\rightarrow K_n(\textbf{x})\xrightarrow{d_n} K_{n-1}(\textbf{x})\rightarrow \cdots\rightarrow K_1(\textbf{x})\xrightarrow{d_1} K_{0}(\textbf{x})\rightarrow 0$$
of finitely generated free modules. For any $R$-module $M$, define $K^{\bullet}(\textbf{x},M)$ (resp., $K_{\bullet}(\textbf{x},M)$ ) to be the complex $\Hom_R(K_{\bullet}(\textbf{x}),M)$ (resp., $K_{\bullet}(\textbf{x})\otimes_RM$). The $p$-th homologies, which is called Koszul cohomology (resp., Koszul homology), is denoted by $\H^p(\textbf{x},M)$ (resp., $\H_p(\textbf{x},M)$), respectively.

Suppose $I$ is an ideal of $R$ generated by  $\textbf{x}$. Then define the Koszul grade of $I$ on $M:$ $$\K.\grade_R(I,M)=\inf\{p\in\mathbb{N}\mid \H^p(\textbf{x},M)\not=0\}.$$
Note that the Koszul grade does not depend on the choice of generating sets of $I$ by \cite[Corollary 1.6.22 and Proposition 1.6.10(d)]{BH93}. For an ideal
$J$ (not necessary finitely generated) of $R$,  the Koszul grade of $J$ on $M:$ $$\K.\grade_R(J,M)=\sup\{\K.\grade_R(I,M)\mid I\ \mbox{is a f. g. subideal of }\ J\}.$$

For a single element $x\in R$. Let $C_{\bullet}(x)$ denote the complex $0\rightarrow R\xrightarrow{d_x} R_x\rightarrow 0$, where $d_x$ the natural localization map. For a sequence $\textbf{x}=x_1,\dots,x_n$ of elements in ring $R$,
$C_{\bullet}(\textbf{x})=C_{\bullet}(x_1)\otimes_RC_{\bullet}(x_2)\cdots\otimes_RC_{\bullet}(x_n)$ the tensor of complexes. For any $R$-module $M$, define $C^{\bullet}(\textbf{x},M)$ (resp.,  $C_{\bullet}(\textbf{x},M)$) to be the complex  $\Hom_R(C_{\bullet}(\textbf{x}),M)$ (resp., $C_{\bullet}(\textbf{x})\otimes_RM$).
 The $p$-th \v{C}ech cohomology (resp., \v{C}ech homology), denoted by $\H^p_{\textbf{x}}(M)$ (resp., $\H_p^{\textbf{x}}(M)$),
  is defined to be the $p$-th homology of $C^{\bullet}(\textbf{x},M)$
 (resp., $C_{\bullet}(\textbf{x},M)$).

Suppose $I$ is  an ideal of $R$ generated by $\textbf{x}$. The \v{C}ech grade of $I$ on $M$ is defined to be $$\mbox{\v{C}}.\grade_R(I,M)=\inf\{p\in\mathbb{N}\mid \H^p_{\textbf{x}}(M)\not=0\}.$$
For an ideal
$J$ (not necessary finitely generated) of $R$,  the Koszul grade of $J$ on $M:$
$$\mbox{\v{C}}.\grade_R(J,M)=\sup\{\mbox{\v{C}}.\grade_R(I,M)\mid I\ \mbox{is a f. g. subideal of }\ J\}.$$

Let $I$ be  an ideal of $R$ and $M$ be an $R$-module. Set $$\Gamma_I(M)=\bigcup\limits_{n\in \mathbb{N}}(0:_MI^n),$$ the set of elements of $M$ annihilated by some power of $I$.
Clearly,   $\Gamma_I(M)=\lim\limits_{\rightarrow} \Hom_R(R/I^n,M).$ Note $\Gamma_I(-)$ is a functor of $R$-modules. The $p$-derived functor of $\Gamma_I(-)$, denoted by  $\mbox{H}_I^p(-)$, is called the $p$-th local cohomology. Certainly,
$$\mbox{H}_I^p(M):= \lim\limits_{\rightarrow} \Ext_R^p(R/I^n,M).$$
 The local cohomology grade of  $I$ on $M$ is defined by $$\mbox{H}.\grade_R(I,M)=\inf\{p\in\mathbb{N}\mid \mbox{H}_I^p(M)\not=0\}.$$

Following by \cite{zw23}, the small finitistic dimension of a ring closely related with the  Ext grade of finitely generated ideals on the ring.
The Ext grade of an ideal $I$ on $M$, which is also denoted by $\E$-$\depth$ in \cite{G89}, is defined by $$\E.\grade_R(I,M)=\inf\{p\in\mathbb{N}\mid \Ext_R^p(R/I,M)\not=0\}.$$
It follows by \cite[Theorem 3.1]{zw23} that  a ring $R$ has $\fPD(R)\leq n$ if and only if $\E.\grade_R(I,R)\leq n$ for any finitely generated ideal $I\not=R.$

\begin{proposition}\label{prop}\cite[Proposition 2.2, Proposition 2.3]{AT09} Let $I$ be an ideal of a ring $R$ and $M$ an $R$-module. Then the following statements hold.
\begin{enumerate}
\item Let $\textbf{y}=y_1,\dots,y_t$ be a regular sequence of elements of $I$ on $M$. Then $$\K.\grade_R(I,M)=t+\K.\grade_R(I,M/\textbf{y}M).$$

\item Let $f:R\rightarrow S$ be a flat ring homomorphism. Then  $$\K.\grade_R(I,M)\leq \K.\grade_R(IS,M\otimes_RS).$$

\item Let $I\subseteq J$ be a pair of ideals of $R$. Then $$\K.\grade_R(I,M)\leq \K.\grade_R(J,M).$$

\item Let $f:R\rightarrow S$ be a ring homomorphism and $N$ an $S$-module. Then
$$\K.\grade_R(I,N)= \K.\grade_S(IS,N).$$

\item Let $f:R\rightarrow S$ be a faithfully flat ring homomorphism. Then  $$\K.\grade_R(I,M)= \K.\grade_R(IS,M\otimes_RS).$$

\item $\K.\grade_R(I,M)=\K.\grade_R(\p,M)$ for some prime ideal $\p$ containing $I$.

\item $\K.\grade_R(I,M)=\mbox{\C}.\grade_R(I,M)$.

\item  $\E.\grade_R(I,M)=\H.\grade_R(I,M)$.

\item If $I$ is finitely generated, then $\K.\grade_R(I,M)=\E.\grade_R(I,M)$.
\end{enumerate}
\end{proposition}

\section{Basic on small finitistic dimensions}
Let $M$ be an $R$-module. Then $M$ is said to have a finite projective resolution, denoted by $M\in\FPR$, if  there exist an integer $n$ and an exact sequence
$$0\rightarrow P_n\rightarrow P_{n-1}\rightarrow \dots\rightarrow P_1\rightarrow P_0\rightarrow M\rightarrow 0$$
with each $P_i$  finitely generated projective. We denote $\Proj^{\leq n}$  to be the class of $R$-modules with projective dimensions at most $n$ in  $\FPR$. In 1989, Glaz \cite{G89} introduced the notion of small   finitistic dimension of a ring $R$.
\begin{definition} \cite[Page 67]{G89}
The small finitistic (projective) dimension of $R$, denoted by $\fPD(R)$, is defined to be the supremum of the projective dimensions of $R$-modules in $\FPR$.
\end{definition}
 Clearly, $\fPD(R)\leq n$ if and only if $\FPR=\Proj^{\leq n}$, and $$\fPD(R)\leq \fpD(R)\leq \FPD(R)\leq \gld(R)$$ for any ring $R$. Note  that  $\fPD(R)$ and  $\fpD(R)$ coincide when $R$ is a Noetherian ring. However, they may vary greatly in non-Noetherian ring settings.
 \begin{example}	Let $R=\prod\limits_{\aleph_n}k$ be a ring of direct product of  $\aleph_n$ ($n<\infty$) copies of a field $k$. Assume that $2^{\aleph_n}=\aleph_m$ with $n+1\leq m<\infty$, then $\fpD(R)=\gld(R)=m+1.$ However, $\fPD(R)=0$ for any $n$ by \cite[Corollary 3.6]{zw23}.
 \end{example}
 \begin{example}	For any $n\in \mathbb{N}^+$, there exists a non-field valuation domain $R$ with $\fpD(R)=\gld(R)=n$ (see \cite{O67}). However, $\fPD(R)=1$ for any valuation domain $R$ by \cite[Corollary 3.7]{zw23}.
 \end{example}	
 The author in this paper and Wang \cite{zw23} characterized small finitistic dimensions in terms of finitely generated semi-regular ideals, tilting modules, cotilting modules of cofinite type and  vaguely associated  prime ideals. Using these, we can give an accurate formula for small finitistic dimensions in terms of Koszul grades.
\begin{theorem}\label{main}
Let $R$ be a ring. Then  $\fPD(R)=\sup\{\K.\grade(\m,R)\mid \m\in\Max(R)\}.$
\end{theorem}
\begin{proof}
	It follows by \cite[Theorem 3.1]{zw23} that $\fPD(R)\leq n$ if and only if any  finitely generated ideal $I$ that satisfies $\Ext_R^i(R/I,R)=0$ for each $i=0,\dots,n$ is $R$, that is $\E.\grade_R(I,R)\leq n$ for any finitely generated  ideal $I\not=R.$ Note that $\K.\grade_R(J,M)=\sup\{\E.\grade_R(I,M)\mid I\ \mbox{is a f. g. subideal of }\ J\}$  by Proposition \ref{prop}(9). Consequently, $\fPD(R)=\sup\{\K.\grade(\m,R)\mid \m\in\Max(R)\}.$
\end{proof}

The following result shows that small finitistic dimension of a ring is less than or equal to its Krull dimension.
\begin{theorem}\label{dim}
	Let $R$ be a ring. Then $\fPD(R)\leq \dim(R)$.
\end{theorem}
\begin{proof} Suppose $\dim(R)=d$. Let $\textbf{x}$ be any finite sequence in $R$. It follows by \cite[Proposition 2.4]{HM07}
that 	$\H^p_{\textbf{x}}(R)=0$ for any $p>d$. So $\mbox{\C}.\grade_R(\langle \textbf{x}\rangle,R)\leq d$. It follows by Proposition \ref{prop}(7) that $\K.\grade_R(\langle \textbf{x}\rangle,R)\leq d$ for any finite sequence \textbf{x}. Hence $\K.\grade(\m,R)\leq d$ for any maximal ideal $\m$ of $R$. Consequently, $\fPD(R)\leq d$ by Theorem \ref{main}.	
\end{proof}
The following example shows that $\fPD(R)$ and $\dim(R)$ may vary greatly.
\begin{example}
It is well-known that, 	for any $n\in \mathbb{N}^+\cup\{\infty\}$, there exists a non-field valuation domain $R$ with $\dim(R)=n$. However, $\fPD(R)= 1$ in this situation.
\end{example}

The following example shows that  $\dim(R)$  may be  less than  $\fpD(R)$.
\begin{example} Let $D=k[x^{1/n}\mid n\geq 1]$ with $k$ a field and $\m=\langle x^{1/n}\mid n\geq 1\rangle$ be its maximal ideal. Set $R=D_\m$. Then $R$ is a valuation domain with $\dim(R)=\fPD(R)=1$. However, $\fpD(R)=\gld(R)=2$ by the proof of \cite[Corollary 2]{O67}.
\end{example}

\section{\fPD\ of polynomial rings}

For a ring $R$, we denote by $R[x]$ the  polynomial ring  over $R$. It is well known that $\gld(R[x])=\gld(R)+1$, $w.\gld(R[x])=w.\gld(R)+1$, and $\FPD(R[x])=\FPD(R)+1$ for any ring $R$ (see \cite[Theorem 3.8.23, Theorem 3.10.3]{fk16}). For small finitistic dimensions, we first have the following result for a general ring.
\begin{proposition}\label{geq}
Let $R$ be a ring. Then $\fPD(R[x])\geq \fPD(R)+1.$
\end{proposition}
\begin{proof} Let $\m$ be a maximal ideal of $R$. Then $\m+xR[x]$ is a maximal ideal of $R[x]$. Consequently,
\begin{align*}
&\fPD(R[x])\\	&=\sup\{\K.\grade_{R[x]}(\M,R[x])\mid \M\in \Max(R[x])\}\\
&\geq \sup\{\K.\grade_{R[x]}(\m+xR[x],R[x])\mid \m\in \Max(R)\}\\
&= \sup\{\K.\grade_{R[x]}(\m+xR[x],R[x]/xR[x])+1\mid \m\in \Max(R)\}\\
&=\sup\{\K.\grade_{R[x]/xR[x]}((\m+xR[x])(R[x]/xR[x]),R)+1\mid \m\in \Max(R)\}\\
&=\sup\{\K.\grade_{R}(\m,R)\mid \m\in \Max(R)\}+1\\
&=\fPD(R)+1.
\end{align*} In conclusion, the result holds.
\end{proof}

Recall that a ring $R$ is called a Hilbert ring, also called a Jacobson ring, if any maximal ideal of $R[X]$ contracts to a maximal ideal of $R$, or equivalently, every prime ideal of $R$ is an intersection of maximal ideals. Note that the  polynomial extension and quotient of Hilbert rings  are also Hilbert rings.

\begin{theorem}\label{poly}
	Let $R$ be a Hilbert ring. Then $\fPD(R[x])=\fPD(R)+1.$
\end{theorem}
\begin{proof} We only need to show $\fPD(R[x])\leq \fPD(R)+1$ by Proposition \ref{geq}.

Let $\M$ be a maximal ideal of $R[x]$. Then there is a maximal ideal $\m$ of $R$ such that $\M\cap R=\m$. So there is a monic polynomial $f$ such that $\M=fR[x]+\m[x]$ and   $\overline{f}:=f+\m[x]$ is irreducible  in $R/\m[x]$.                   by \cite[Exercise 1.50]{fk16}. So $f$ is a non-zero-divisor in $R[x]$.  It follows by Proposition \ref{prop} that
\begin{align*}
&\K.\grade_{R[x]}(\M,R[x])\\	&=\K.\grade_{R[x]}(fR[x]+\m[x],R[x])\\
&=\K.\grade_{R[x]}(fR[x]+\m R[x],R[x]/fR[x])+1\\
&=\K.\grade_{R[x]/fR[x]}(((fR[x]+\m R[x])/fR[x]),R[x]/fR[x])+1\\
&=\K.\grade_{R[x]/fR[x]}(\m (R[x]/fR[x]),R[x]/fR[x])+1\\
&= \K.\grade_R(\m,R[x]/fR[x])+1\\
&= \K.\grade_{R[x]}(\m[x],R[x]/fR[x])+1.
\end{align*}
 We consider the following long exact sequence of $R[x]$-modules:
$$\cdots\rightarrow H^{j}(\m,R[x])\xrightarrow{\times f}H^{j}(\m,R[x])\rightarrow H^{j}(\m,R[x]/fR[x])\rightarrow H^{j+1}(\m,R[x])\xrightarrow{\times f}H^{j+1}(\m,R[x])\rightarrow \cdots$$
Since $f$ is monic and $\overline{f}$ is irreducible,   multiplying $f$ is a monomorphism. So we have
\begin{align*}	
&\K.\grade_{R[x]}(\m[x],R[x]/fR[x])\\
&\leq \K.\grade_{R[x]}(\m[x],R[x])\\
&=\K.\grade_R(\m,R).
\end{align*}
The last equation holds since $R[x]$ is a faithfully flat $R$-module. Hence $\K.\grade_{R[x]}(\M,R[x])\leq \K.\grade_R(\m,R)+1$
Consequently, $\fPD(R[x])\leq \fPD(R)+1$.
\end{proof}

\begin{lemma}\label{depthht}
Let $R$ be a Noetherian ring, and $\p\subsetneq \q$
be prime ideals of $R$ such that $ht(\q)=ht(\p)+1$. Then $$\K.\grade_R(\q,R)\leq\K.\grade_R(\p,R)+1.$$
\end{lemma}
\begin{proof} Suppose $R$ is a Noetherian ring. Then $\K.\grade_R(\q,R)=\grade(\q,R)$ and $\K.\grade_R(\p,R)=\grade(\p,R)$. Let $\{x_1,\dots,x_n\}$ be a maximal $R$-sequence in $\p$. On the contrary, there is an $R$-sequence $\{x_1,\dots,x_n,x_{n+1},x_{n+2}\}$ in $\q$  by \cite[Theorem 1.2.5]{BH93}. Then $\overline{x_{n+1}},\overline{x_{n+2}}$ is an $R/\p$-sequence in $\q/\p$ by \cite[Proposition 1.2.10(d)]{BH93}. Hence  $\grade(\q/\p,R/\p)\geq 2,$    which is impossible by \cite[Proposition 1.2.14]{BH93} since $ht_{R/\p}(\q/\p)=1$.
\end{proof}

\begin{theorem}\label{scr}
	Let $R$ be a Noetherian ring. Then $\fPD(R[x])=\fPD(R)+1.$
\end{theorem}
\begin{proof} We only need to show $\fPD(R[x])\leq \fPD(R)+1$ by Proposition \ref{geq}.

Let $\M$ be a maximal ideal of $R[x]$. Suppose $\M\cap R=\p$. Then $\p[x]\subsetneq \M$. Then $ht(\M)=ht(\p[x])+1$ by  \cite[Theorem  1.8.16]{fk16}.
 So it follows by Proposition \ref{prop} and Lemma \ref{depthht} that \begin{align*}
 	&\K.\grade_{R[x]}(\M,R[x])\\	&\leq\K.\grade_{R[x]}(\p[x],R[x])+1\\
 	&=\K.\grade_{R}(\p ,R)+1\\
 	&\leq\K.\grade_{R}(\m ,R)+1.
 \end{align*}
 where $\m$ is a maximal ideal contains $\p$ and the firs equality follows by Proposition \ref{prop}(6). Consequently, $\fPD(R[x])\leq \fPD(R)+1$.	
\end{proof}

\begin{remark} We wonder whether the following statement holds for all rings $R$:

{\it Let $R$ be a  ring, and $\p\subsetneq \q$
	be prime ideals of $R$ such that $ht(\q)=ht(\p)+1$. Then} $$\K.\grade_R(\q,R)\leq\K.\grade_R(\p,R)+1.$$
Suppose $R$ is a Noetherian ring. Then the above statement holds true by   Lemma \ref{depthht}. If the above statement holds true for  all rings $R$,  then we always have $\fPD(R[x])=\fPD(R)+1$ by the proof of Theorem \ref{scr}.
\end{remark}

\section{\fPD\ of formal power series rings}

For a ring $R$, we denote by $R[[x]]$ the formal power series ring  over $R$. It is well-known that the maximal ideal of $R[[x]]$ is corresponding with that of $R$ one by one:

\begin{lemma}\cite[Thoerem 2]{B81}
Let $R$ be a ring. Then $\Max(R[[x]])=\{\m+\langle x\rangle \mid \m\in \Max(R)\}.$
\end{lemma}

The studies of homological dimension of  formal power series rings have attracted many algebraists.  Auslander and Buchsbaum \cite{AB581} showed   if $R$ is a
Noetherian ring, then $\gld(R[[x]]) = \gld(R) + 1$.  Later,
Jondrup and Small \cite{J74} obtained $w.\gld(R[[x]]) = w.\gld(R) + 1$ in the case where $R[[x]]$ is a coherent ring.

\begin{proposition}
	Let $R$ be a ring . Then $\fPD(R[[x]])\geq \fPD(R)+1.$
\end{proposition}
 \begin{proof}
 	It is similar with the proof of Proposition \ref{geq}, and so we omit it.
 \end{proof}

\begin{theorem}\label{pow}
Let $R$ be a ring such that $R[[x]]$ is a coherent ring. Then $\fPD(R[[x]])=\fPD(R)+1.$
\end{theorem}
\begin{proof}    Since $R[[x]]$ is a coherent ring, so is $R$. And thus $R[[x]]\cong \prod\limits_{i=1}^{\infty}R$ is a flat $R$-module. Since $R[[x]]$ contains a faithfully flat module $R$, $R[[x]]$ is also a faithfully flat $R$-module. Then for any maximal ideal $\m$, we have $$\K.\grade_R(\m,R)= \K.\grade_{R[[x]]}(\m R[[x]],R[[x]])$$ by Proposition \ref{prop}(5).  Let $I=\langle f_1,\dots,f_n\rangle$ be a finitely generated ideal in $\m R[[x]]$. Set $\textbf{y}= f_1,\dots,f_n$ and $\textbf{x}=f_1,\dots,f_n,x$. We consider the following long exact sequence of $R[[x]]$-modules:
$$\cdots\rightarrow H^{j}(\textbf{x},R[[x]])\xrightarrow{x}H^{j}(\textbf{x},R[[x]])\rightarrow H^{j+1}(\textbf{y},R[[x]])\rightarrow H^{j+1}(\textbf{x},R[[x]])\rightarrow \cdots$$
By \cite[Lemma 3.7]{AT09}, $H^{j}(\textbf{x},R[[x]])$ is finitely generated. Note $x$ belongs to the Jacobson radical of $R[[x]]$. By Nakayama's Lemma, we have
$$\K.\grade_{R[[x]]}(I+\langle x\rangle,R[[x]])=\K.\grade_{R[[x]]}(I,R[[x]])+1.$$ Note that  $\m+\langle x\rangle=\m R[[x]]+\langle x\rangle$. It follows  that
\begin{align*}
	&\K.\grade_{R[[x]]}(\m+\langle x\rangle,R[[x]])\\	&=\K.\grade_{R[[x]]}(\m R[[x]]+\langle x\rangle,R[[x]])\\
	&=\K.\grade_{R[[x]]}(\m R[[x]],R[[x]])+1\\
	&= \K.\grade_R(\m,R)+1.
\end{align*}
Therefore,  $\fPD(R[[x]])=\fPD(R)+1.$
\end{proof}

\begin{remark}
Note that the condition ``$R[[x]]$ is a coherent ring'' in Theorem \ref{pow} is not necessary. Indeed, let $R$ be a von Neumann regular ring with $R[[x]]$ not coherent (see \cite[Page 280]{G89}). Note that $R[[x]]$ is always a B\'{e}zout ring by \cite[Theorem 8.1.4]{G89} and  \cite[corollary 4.4]{B15}. So $R[[x]]$ is a $\DW$-ring by \cite[Exercise 6.11(2)]{fk16}. Hence $\fPD(R[[x]])=1=\fPD(R)+1$ by \cite[Corollary 3.7]{zw23}. It is an interesting question that what's the relationship  between $\fPD(R[[x]])$ and $\fPD(R)$ for a general ring $R$.
\end{remark}

\section{\fPD\ of trivial extensions}

Let $R$ be a ring and $M$ be an $R$-module. Then the \emph{trivial extension} of $R$ by $M$, denoted by $R(+)M$, is equal to $R\bigoplus M$ as $R$-modules with coordinate-wise addition and multiplication $(r_1,m_1)(r_2,m_2)=(r_1r_2,r_1m_2+r_2m_1)$. It is easy to verify that $R(+)M$ is a commutative ring with identity $(1,0)$. The maximal ideal of $R(+)M$ is corresponding with that of $R$ one by one:

\begin{lemma}\cite[Theorem 3.2]{AW09}
	Let $R$ be a ring and $M$ an $R$-module. Then $\Max(R(+)M)=\{\m(+)M \mid \m\in \Max(R)\}.$
\end{lemma}

\begin{theorem}\label{trivext}
	Let $R$ be a ring and $M$ an $R$-module. Then $$\fPD(R(+)M)=\sup\{\min\{\K.\grade_{R}(\m,R),\K.\grade_{R}(\m,M)\}\mid \m\in\Max(R)\}\leq \fPD(R).$$
\end{theorem}
\begin{proof} Let $\m$ be a maximal ideal of $R$.  Consider the natural embedding map $f:R\rightarrow R(+)M$. It follows by Proposition \ref{prop}(4) that
\begin{align*}
&\K.\grade_{R(+)M}(\m(+)\m M,R(+)M)\\	&=\K.\grade_{R(+)M}(\m(R(+)M),R(+)M)\\
&=\K.\grade_{R}(\m,R(+)M)\\
&=\min\{\K.\grade_{R}(\m,R),\K.\grade_{R}(\m,M)\}.
\end{align*}
Note that $\m(+)M$ is the unique prime ideal that contains  $\m(+)\m M$. So we have $$\K.\grade_{R(+)M}(\m(+) M,R(+)M)=\K.\grade_{R(+)M}(\m(+)\m M,R(+)M)$$ by Proposition \ref{prop}(6). Consequently,
\begin{align*}
	&\fPD(R(+)M)\\	&=\sup\{\min\{\K.\grade_{R}(\m,R),\K.\grade_{R}(\m,M)\}\mid \m\in\Max(R)\}\\
	&\leq \sup\{\K.\grade_{R}(\m,R)\}\mid \m\in\Max(R)\}\\
	&= \fPD(R).
\end{align*}
In conclusion, the result holds.
\end{proof}

\begin{corollary}\label{fdtQ}
Let $D$ be a non-field integral domain with $Q$ its quotient field. Then
\begin{center}
	$\fPD(D(+)Q)=\fPD(D)$, and\   $\fPD(D(+)Q/D)=\fPD(D)-1.$
\end{center}
\end{corollary}
\begin{proof}
	Since $Q$ is injective and torsion-free, we have $K.\grade_{D}(\m,Q)=\infty$ for any maximal ideal $\m\in \Max(R)$ since $D$ is not a field.  Hence, $\fPD(D)= \fPD(D(+)Q)$.
	
	Note that for any $n\geq 0$ and any nonzero ideal $I$ of $D$, we have $\Ext_D^{n}(D/I,Q/D)\cong \Ext_D^{n+1}(D/I,D)$.
	So, $\fPD(D(+)Q/D)=\fPD(D)-1.$	
\end{proof}

Recall from \cite{G69} that a commutative ring $R$ is said to be a \Prufer\ ring provided that every finitely generated regular ideal is invertible. Obviously, every total ring of quotients (i.e. any non-zero-divisor is invertible) is \Prufer. In \cite[Problem 1]{CFFG14}, Cahen et al. posed the following two open questions:
\begin{itemize}
	\item {\bf Problem 1a:} Let $R$ be a \Prufer\ ring. Is $\fPD(R)\leq 1$?
	\item {\bf Problem 1b:} Let $R$ be a total ring of quotients. Is $\fPD(R)=0$?
\end{itemize}
Recently, Wang et al. \cite{wzcc20,fkxs20} obtained a total ring of quotients $R$ with $\fPD(R)>1$ getting a negative answer to these two open questions. Latter, the author in this paper and wang \cite{zw23} shows that, for any $n\in \mathbb{N}$, there exists a  total ring of quotients $R$ satisfying $\fPD(R)=n$. Now, we give an example to show that the small finitistic dimension of a  total ring of quotients can even be infinite.

\begin{example}\label{glazinf} Let $D$ be the Nagata's bad Noetherian domain given in \cite[Appendix,
	Example 1]{N62} with $Q$ its quotient field. Then $\fPD(D)=\infty $ by \cite[Example 3.5]{zw23}. Set $R=D(+)Q/D$. Then $R$ is a  total ring of quotients by \cite[Theorem 3.5]{AW09}. However, $\fPD(R)=\infty$ by Corollary \ref{fdtQ}.
\end{example}

\section{\fPD\ of amalgamations}

Let $f:A\rightarrow B$ be a ring homomorphism and $J$ an ideal of $B$. Following from \cite{DFF09} that the  \emph{amalgamation} of $A$ with $B$ along $J$ with respect to $f$, denoted by $A\bowtie^fJ$, is defined as $$A\bowtie^fJ=\{(a,f(a)+j) \mid a\in A,j\in J\},$$  which is  a subring of of $A \times B$.  By \cite[Proposition 4.2]{DFF09}, $A\bowtie^fJ$  is the pullback $\widehat{f}\times_{B/J}\pi$,
where $\pi:B\rightarrow B/J$ is the natural epimorphism and $\widehat{f}=\pi\circ f$:
$$\xymatrix@R=20pt@C=25pt{
	A\bowtie^fJ\ar[d]^{p_B}\ar[r]_{p_A}& A\ar[d]^{\widehat{f}}\\
	B\ar[r]^{\pi}&B/J. \\
}$$

Let $\p$ be a prime ideal of $A$ and $\q$ be a prime ideal of $B$. Set
\begin{enumerate}
\item $\p'^f:=\p\bowtie^fJ=\{(p,f(p)+j)\mid p\in\p,f\in J\};$
\item $\overline{\q}^f:=\{(a,f(a)+j)\mid a\in A,f(a)+j\in \q\}.$
\end{enumerate}

\begin{lemma}\label{a-pri}\cite[Proposition 2.6]{DFM10}
Let $f:A\rightarrow B$ be a ring homomorphism and $J$ an ideal of $B$. Then $$\Spec(A\bowtie^fJ)=\{\p'^f\mid \p\in \Spec(A)\}\cup \{\overline{\q}^f\mid \q\in \Spec(B)- V(J)\},$$  $$\Max(A\bowtie^fJ)=\{\p'^f\mid \p\in \Max(A)\}\cup \{\overline{\q}^f\mid \q\in \Max(B)- V(J)\}.$$
So if $J\subseteq \Nil(B)$  then $\Spec(A\bowtie^fJ)=\{\p'^f\mid \p\in \Spec(A)\}$, and if $J\subseteq \Rad(B)$ then $\Max(A\bowtie^fJ)=\{\p'^f\mid \p\in \Max(A)\}.$
\end{lemma}

\begin{theorem}\label{amg}
Let $f:A\rightarrow B$ be a ring homomorphism and $J$ an ideal of $B$ contained in $\Nil(B)$. Then $$\fPD(A\bowtie^fJ) =\sup\{\min\{\K.\grade_{R}(\m,A),\K.\grade_{A}(\m,J)\}\mid \m\in\Max(A)\}\leq \fPD(A).$$
\end{theorem}
\begin{proof} Consider the natural map $\alpha:A\rightarrow A\bowtie^fJ$ by $\alpha(a)=(a,f(a))$ for any $a\in A$. Let $\m$ be a maximal ideal of $A$. It follows by Proposition \ref{prop}(4) that
	\begin{align*}
		&\K.\grade_{A\bowtie^fJ}(\m\bowtie^ff(\m)J),A\bowtie^fJ)\\	&=\K.\grade_{A\bowtie^fJ}(\m(A\bowtie^fJ),A\bowtie^fJ)\\
		&=\K.\grade_{A}(\m,A\bowtie^fJ)\\
		&=\min\{\K.\grade_{A}(\m,A),\K.\grade_{A}(\m,J)\}.
	\end{align*}
The last equation follows by that $A\bowtie^fJ\cong A\oplus J$ as $A$-modules.

Since  $J$ is  contained  in $\Nil(B)$, $\m'^f$ is the unique prime ideal that contains  $\m\bowtie^ff(\m)J$ by Lemma \ref{a-pri}. So we have $$\K.\grade_{A\bowtie^fJ}(\m'^f,A\bowtie^fJ)=\K.\grade_{A\bowtie^fJ}(\m\bowtie^ff(\m)J,A\bowtie^fJ)$$ by Proposition \ref{prop}(6). Consequently,
	\begin{align*}
		&\fPD(A\bowtie^fJ)\\	&=\sup\{\min\{\K.\grade_{A}(\m,A),\K.\grade_{A}(\m,J)\}\}\\
		&\leq \sup\{\K.\grade_{A}(\m,A)\}\mid \m\in\Max(A)\}\\
		&= \fPD(A).
	\end{align*}
In conclusion, the result holds.
\end{proof}
\begin{remark}
Note that the condition that ``$J$ an ideal of $B$ in $\Nil(B)$'' in Theorem \ref{amg} cannot be omitted. Indeed, since $A[[x]]\cong A\bowtie^ixA[[x]]$ where $i:A\hookrightarrow A[[x]]$ is the natural embedding map,  $\fPD(A\bowtie^ixA[[x]])=\fPD(A)+1$ when $A[[x]]$ is coherent (see Theorem \ref{pow}).
\end{remark}

\begin{remark}
The following example shows that $\fPD(A\bowtie^fJ)$ can be strictly less than $\fPD(A)$. Indeed, let $A$ be a ring and $M$ an $A$-module. Set $B=A(+)M$, $i:A\rightarrow B$ the natural embedding map and $J=0(+)M$. Then $A\bowtie^iJ\cong A(+)M$. It follows by Corollary \ref{fdtQ} that $\fPD(A\bowtie^iJ)$ is strictly less than $\fPD(A)$ in the case that  $A$ is a non-field integral domain and $M=Q/A$ with $Q$ its quotient field.
\end{remark}






\end{document}